\documentclass[secthm,seceqn,amsthm,ussrhead,reqno, 12pt]{amsart}
\usepackage[utf8]{inputenc}
\usepackage[english]{babel}
\usepackage[symbol]{footmisc}
\usepackage{amssymb,amsmath,amsthm,amsfonts,xcolor,enumerate,hyperref,comment,longtable,cleveref}

\usepackage{times}
\usepackage{cite}
\usepackage{pdflscape}
\usepackage{ulem}
 
\usepackage{tikz}
\usepackage{hyperref}
\usepackage{cancel}
\usepackage{stmaryrd}

\newtheorem{theorem}{Theorem}

\newtheorem{lemma}[theorem]{Lemma}

\newtheorem{definition}[theorem]{Definition}

\newtheorem{rem}[theorem]{Remark}

\crefrangeformat{equation}{#3(#1)#4--#5(#2)#6}

\crefname{theorem}{theorem}{Theorems}
\crefname{lem}{Lemma}{Lemmas}
\crefname{cor}{Corollary}{Corollaries}
\crefname{prop}{Proposition}{Propositions}
\crefname{proposition}{Proposition}{Propositions}
\crefname{defn}{Definition}{Definitions}
\crefname{exm}{Example}{Examples}
\crefname{rem}{Remark}{Remarks}
\crefname{section}{Section}{Sections}
\crefname{equation}{\unskip}{\unskip}
\crefname{enumi}{\unskip}{\unskip}

\newcommand{\gen}[1]{\langle #1\rangle}

\newcommand{\red}[1]{\textcolor{black}{#1}}

\begin{document}

\noindent{\Large 
Transposed Poisson structures on Witt-type algebras}
 \footnote{
The  first part of work is supported by 
FCT   UIDB/00212/2020, UIDP/00212/2020, and 2022.02474.PTDC.
The second part of this work is supported by the Russian Science Foundation under grant 22-71-10001.
} 

	\bigskip
	
	 \bigskip

\begin{center}	
	{\bf
		Ivan Kaygorodov\footnote{CMA-UBI, Universidade da Beira Interior, Covilh\~{a}, Portugal; \
  Moscow Center for Fundamental and Applied Mathematics,      Russia; 
  \    Saint Petersburg  University, Russia;\  kaygorodov.ivan@gmail.com},
   Abror Khudoyberdiyev\footnote{V.I.Romanovskiy Institute of Mathematics Academy of Science of Uzbekistan; National University of Uzbekistan; \ 
khabror@mail.ru} \&
Zarina Shermatova\footnote{V.I.Romanovskiy Institute of Mathematics Academy of Science of Uzbekistan; Kimyo International University in Tashkent, Uzbekistan; \ 
z.shermatova@mathinst.uz}}   
\end{center}

 \bigskip
  
\noindent {\bf Abstract.}
{\it 
We compute  $\frac{1}{2}$-derivations on the deformative Schr\"{o}dinger-Witt  \footnote{We refer the notion of \textquotedblleft Witt\textquotedblright  \   algebra to the simple Witt algebra and the notion of \textquotedblleft Virasoro\textquotedblright  \  algebra to the central extension of the simple Witt algebra.}
 algebra, on not-finitely graded Witt algebras $W_n(G)$, 
 and on not-finitely graded Heisenberg-Witt algebra $HW_n(G)$.
We classify all transposed Poisson structures on
such algebras. 

}



 \bigskip
\section*{Introduction} 

Since their origin in the 1970s in Poisson geometry, Poisson algebras have appeared in several areas of mathematics and physics, such as algebraic geometry, operads, quantization theory, quantum groups, and classical and quantum mechanics. One of the natural tasks in the theory of Poisson algebras is the description of all such algebras with a fixed Lie or associative part.
	Recently, Bai, Bai, Guo,  and Wu have introduced a dual notion of the Poisson algebra~\cite{bai20}, called a transposed Poisson algebra, by exchanging the roles of the two multiplications in the Leibniz rule defining a Poisson algebra. A transposed Poisson algebra defined this way not only shares some properties of a Poisson algebra, such as the closedness under tensor products and the Koszul self-duality as an operad but also admits a rich class of identities \cite{kms,bai20,fer23,lb23}. It is important to note that a transposed Poisson algebra naturally arises from a Novikov-Poisson algebra by taking the commutator Lie algebra of its Novikov part \cite{bai20}.
 	Any unital transposed Poisson algebra is
	a particular case of a ``contact bracket'' algebra 
	and a quasi-Poisson algebra.
 Each transposed Poisson algebra is a 
 commutative  Gelfand-Dorfman algebra \cite{kms}
 and it is also an algebra of Jordan brackets \cite{fer23}.
Transposed Poisson algebras are related to weak Leibniz algebras \cite{dzhuma}.
	In a recent paper by Ferreira, Kaygorodov, Lopatkin
	a relation between $\frac{1}{2}$-derivations of Lie algebras and 
	transposed Poisson algebras has been established \cite{FKL}. 	These ideas were used to describe all transposed Poisson structures 
	on  Witt and Virasoro algebras in  \cite{FKL};
	on   twisted Heisenberg-Virasoro,   Schr\"odinger-Virasoro  and  
	extended Schr\"odinger-Virasoro algebras in \cite{yh21};
	on Schr\"odinger algebra in $(n+1)$-dimensional space-time in \cite{ytk};
\red{on solvable Lie algebra with filiform nilradical} in \cite{aae23};
on oscillator Lie algebras in \cite{kkh24};
	on Witt type Lie algebras in \cite{kk23};
	on generalized Witt algebras in \cite{kkg23}; 
 Block Lie algebras in \cite{kk22,kkg23}
  and
on    Lie incidence algebras (for all references, see the survey \cite{k23}).
  
 \medskip

The description of transposed Poisson structures obtained on the Witt algebra \cite{FKL} opens the question of finding algebras related to the Witt algebra
which admit nontrivial transposed Poisson structures. 
So, some algebras related to Witt algebra are studied in \cite{kk23,kk22,kkg23}.
The present paper is a continuation of this research. 
Specifically, we describe transposed Poisson structures on 
the deformative Schr\"{o}dinger-Witt algebra
${\mathcal W}(a,b,s),$  algebras $W_n(G)$ and algebras $HW_n(G).$
The algebras ${\mathcal W}(a,b,s)$ contain a subalgebra isomorphic to the well-known $W$-algebra  ${\mathcal W}(a, b).$ 
  The algebra   ${\mathcal W}(0,0,0)$ was introduced in \cite{h94} and later 
   it was generalized in \cite{ru06}.
   Their structures   have been widely studied in 
   \cite{jw15,ru06,wl12,jt20}.
Namely, 
derivations and automorphism group of  ${\mathcal W}(a,b,s)$ are described in  \cite{jw15,wl12};
  all biderivations   and commutative post-Lie algebra structures on  ${\mathcal W}(a,b,s)$ are obtained in \cite{jt20};
  graded post-Lie algebra structures and homogeneous Rota-Baxter operators on  
  ${\mathcal W}(0,0,0)$ are given in \cite{xt21} and so on.
The algebra $W_{-1}(G)$ was independently introduced by  Passman \cite{p98} and  Xu \cite{xu97}. It naturally appeared in the theory of integrable systems in  \cite{25}.
All one-dimensional central extensions of $W_{-1}(G)$ are obtained in \cite{su}.
Derivations of  $W_{-1}(G)$  with coefficients in tensor modules were found in \cite{gpw}.
The algebra  $W_1(G)$ firstly was defined in \cite{chs15}.
Derivations, automorphisms,  one-dimensional extensions, and the second cohomology group of  $W_1(G)$ were defined also in  \cite{chs15}.  Lie bialgebras structure on  $W_1(G)$ was described in \cite{lb15}. The algebra
$W_0(G)$ is an algebra of Witt type defined in \cite{Yu97}. 
We introduce the most general concept generalizing of yearly considered algebras $W_1(G),$ $W_0(G)$ and $W_{-1}(G).$ Namely, 
for each $n\in \mathbb Z$ we define an algebra $W_n(G)$ and studying its properties.
 The algebra $HW_{-1}(G)$ was firstly introduced in \cite{fzy15}, and
$HW_{-1}(G)$ is constructed as the direct sum of $W_{-1}(G)$ and a module over $W_{-1}(G).$ Similarly, we define and study of algebras $HW_n(G)$   for an arbitrary $n \in \mathbb Z.$

 \medskip

Summarizing, we prove that algebras ${\mathcal W}(a,b,s)$ admit nontrivial transposed Poisson structures only for $b=-1$;
algebras $W_n(G)$ admit nontrivial transposed Poisson structures   and algebras $HW_n(G)$ do not admit nontrivial transposed Poisson structures.

\section{Preliminaries}\ \ \ \ \ \ \ \ \

In this section, we recall some definitions and known results for studying transposed Poisson structures. Although all algebras and vector spaces are considered over the complex field, many results can be proven over other fields without modifications of proofs. The notation $\gen{S}$ means the $\mathbb C$-subspace generated by $S$.

\begin{definition} 
Let $\mathfrak {L}$ be a vector space equipped with two nonzero bilinear operations $\cdot$ and $[\cdot, \cdot]$. The triple $(\mathfrak {L}, \cdot, [\cdot, \cdot])$ is called a transposed Poisson algebra if $(\mathfrak {L}, \cdot)$ is a commutative associative algebra and $(\mathfrak {L}, [\cdot, \cdot])$ is a Lie algebra that satisfies the following compatibility condition
$$
\begin{array}{c}
2z \cdot [x,y]=[z \cdot x, y]+[x, z \cdot y].
\end{array}
$$
\end{definition}

\begin{definition} 
Let $(\mathfrak{L}, [\cdot, \cdot])$ be a Lie algebra. A transposed Poisson structure on $(\mathfrak {L}, [\cdot, \cdot])$ is a commutative associative multiplication $\cdot$ in $\mathfrak {L}$ which makes $(\mathfrak {L}, \cdot, [\cdot, \cdot])$ a transposed Poisson algebra.
\end{definition}

\begin{definition} 
Let $(\mathfrak {L}, [\cdot, \cdot])$ be a Lie algebra, $\varphi: \mathfrak {L}\rightarrow \mathfrak {L} $ be a linear map. Then $\varphi$ is a $\frac 12$-derivation if it satisfies
$$
\begin{array}{c}
\varphi ([x,y])= \frac 12 \big([\varphi (x), y]+[x, \varphi (y)]\big).
\end{array}
$$
\end{definition}

Observe that $\frac 12$-derivations are a particular case of $\delta$-derivations introduced by Filippov in 1998 (for references about the  study of $\frac{1}{2}$-derivations, see \cite{k23}) and 
recently the notion of $\frac{1}{2}$-derivations of algebras was generalized to 
 $\frac{1}{2}$-derivations from an algebra to a module \cite{zz}.  
The main example of $\frac 12$-derivations is the multiplication by an element from the ground field. Let us call such $\frac 12$-derivations as  {trivial
$\frac 12$-derivations.} It is easy to see that $[\mathfrak {L},\mathfrak {L}]$ and $\operatorname{Ann}(\mathfrak {L})$ are invariant under any $\frac 12$-derivation of $\mathfrak {L}$.

Let $G$ be an abelian group, $\mathfrak{L}=\bigoplus \limits_{g\in G}\mathfrak{L}_{g}$ be a $G$-graded Lie algebra.
We say that a $\frac 12$-derivation $\varphi$ has degree $g$ (and denopted by $\deg(\varphi)$) if $\varphi(\mathfrak{L}_{h})\subseteq \mathfrak{L}_{g+h}$.
Let $\triangle(\mathfrak{L})$ denote the space of $\frac 12$-derivations and write $\triangle_{g}(\mathfrak{L})=\{\varphi \in \triangle(\mathfrak{L}) \mid \deg(\varphi)=g\}$.
The following trivial lemmas are useful in our work.

\begin{lemma}\label{l01}
Let $\mathfrak{L}=\bigoplus \limits_{g\in G}\mathfrak{L}_{g}$ be a $G$-graded Lie algebra. Then
$
\triangle(\mathfrak{L})=\bigoplus \limits_{g\in G}\triangle_{g}(\mathfrak{L}).
$
\end{lemma}

\begin{lemma}\label{l1} (see {\rm\cite{FKL}})
Let $(\mathfrak {L}, \cdot, [\cdot, \cdot])$ be a transposed Poisson algebra and $z$ an arbitrary element from $\mathfrak {L}$.
Then the left multiplication $L_{z}$ in the   commutative associative algebra $(\mathfrak {L}, \cdot)$ gives a $\frac 12$-derivation of the Lie algebra $(\mathfrak {L}, [\cdot, \cdot])$.
\end{lemma}

\begin{lemma}  (see {\rm\cite{FKL}})
Let $\mathfrak {L}$ be a Lie algebra without non-trivial $\frac 12$-derivations. Then every transposed Poisson structure defined on $\mathfrak {L}$ is trivial.
\end{lemma}

\section{Transposed Poisson   structures on the deformative Schr\"{o}dinger-Witt algebras}\ \ \ \ \ \ \ \ \

Let us give the definition of the algebra ${\mathcal W}(a,b,s),$  
where  $a,b\in \mathbb{C}$ and $s\in \{0,\frac{1}{2}\}.$ 
${\mathcal W}(a,b,s)$ is spanned by  $\{L_m, I_m, Y_{m+s} \ | \ m \in \mathbb{Z}\}$ and 
the  multiplication table is given by the following nontrivial relations:
\begin{longtable}{lllll}
$[L_m,L_n]$&$=$&$(n-m)L_{m+n},$\\
$[L_m,I_n]$&$=$&$(n+bm+a)I_{m+n},$ \\
$[L_m,Y_{n+s}]$&$=$&$\left(n+s+\frac{(b-1)m+a}{2}\right)Y_{m+n+s},$\\
$[Y_{m+s},Y_{n+s}]$&$=$&$(n-m)I_{m+n+2s}.$\\
\end{longtable}

 Thanks to \cite{jt20}, we  have 
 ${\mathcal W}(a+1,b,0) \cong {\mathcal W}(a,b,\frac{1}{2}).$ Hence,  we only have to  consider the case of $s=\frac{1}{2}$.
Set $M_n= \langle L_n,I_n\rangle$ and $M_{n+\frac{1}{2}}=\langle  Y_{n+\frac{1}{2}}\rangle.$
Then ${\mathcal W}(a,b,\frac{1}{2})$ is $\frac{1}{2}\mathbb{Z}$-graded, i.e.,

\begin{center}
    ${\mathcal W}(a,b,\frac{1}{2})=\big(\bigoplus\limits_{n\in \mathbb{Z}}M_n\big)\bigoplus\big(\bigoplus\limits_{n\in \mathbb{Z}}M_{n+\frac{1}{2}}\big
    ).$
\end{center}
Hence $\Delta({\mathcal W}(a,b,\frac{1}{2}))$ has a natural $\frac{1}{2}\mathbb{Z}$-grading, i.e.,
\begin{center}$\Delta({\mathcal W}(a,b,\frac{1}{2}))=
\big(\bigoplus\limits_{n\in \mathbb{Z}}\Delta_n({\mathcal W}(a,b,\frac{1}{2}))\big)\bigoplus
\big(\bigoplus\limits_{n\in \mathbb{Z}}\Delta_{n+\frac{1}{2}}({\mathcal W}(a,b,\frac{1}{2}))\big).$
\end{center}

Note that ${\mathcal W}(a,b,\frac{1}{2})$ contains a subalgebra $\langle  L_m,I_m \ |\ m\in\mathbb{Z}\rangle,$ which isomorphic to the well-known algebra ${\mathcal W}(a,b)$   with the multiplication table:
$$[L_m,L_n]=(n-m)L_{m+n}, \quad [L_m,I_n]=(n+bm+a)I_{m+n}.$$

The description of $\frac{1}{2}$-derivations of Lie algebra ${\mathcal W}(a,b)$ is given in the following theorem. 

\begin{theorem}[see, \cite{FKL}] \label{dif}
If $b\neq-1,$ then there are no non-trivial $\frac{1}{2}$-derivations of ${\mathcal W}(a,b).$
Let $\varphi$ be a $\frac{1}{2}$-derivation of the algebra ${\mathcal W}(a,-1),$ then there are two finite sets of elements from the basic field $\{\alpha_t\}_{t\in\mathbb{Z}}$ and $\{\beta_t\}_{t\in\mathbb{Z}},$
such that
$$\varphi(L_m)=\sum_{t\in\mathbb{Z}}\alpha_tL_{m+t}+\sum_{t\in\mathbb{Z}}\beta_tI_{m+t},  \quad
\varphi(I_m)=\sum_{t\in\mathbb{Z}}\alpha_tI_{m+t}.$$
\end{theorem}

Now we will compute $\frac{1}{2}$-derivations of the algebras ${\mathcal W}(a,b,\frac{1}{2}).$
Summarizing the results from lemmas \ref{12even} and \ref{lemm10}, we conclude that
if $b\neq-1,$ then ${\mathcal W}(a,b,\frac{1}{2})$ does not have nontrivial $\frac{1}{2}$-derivations. 
On the other side,
Theorem \ref{main2} gives the full description of nontrivial  $\frac{1}{2}$-derivations of ${\mathcal W}(a,-1,\frac{1}{2}).$

\begin{lemma}\label{12even}  
If $b\neq-1,$ then
$\Delta_0({\mathcal W}(a,b,\frac{1}{2}))=\langle {\rm Id} \rangle$ and
 $\Delta_{j \in \mathbb{Z}\setminus\{0\}}({\mathcal W}(a,b,\frac{1}{2}))=0.$

\end{lemma}

\begin{proof}
Suppose that $b\neq-1$ and  $\varphi_j\in \Delta_j({\mathcal W}(a,b,\frac{1}{2}))$    be a homogeneous $\frac{1}{2}$-derivation. In this case, we have
\begin{equation}\label{der}
\varphi_j(M_n)\subseteq M_{n+j}, \quad \varphi_j (M_{n+\frac{1}{2}})\subseteq M_{n+j+\frac{1}{2}}.
\end{equation}

\noindent By Theorem \ref{dif}, we get
\begin{equation}\label{der1}
\varphi_j(L_n)=\delta_{j,0}\lambda L_{n}, \quad \varphi_j(I_n)=\delta_{j,0}\lambda I_{n}. 
\end{equation}
By (\ref{der}), we can assume that
\begin{equation}\label{der4} \varphi_j(Y_{n+\frac{1}{2}})=a_{j,n} Y_{n+j+\frac{1}{2}}.
\end{equation}

\begin{enumerate}
    \item[(1)] If $j\neq0,$ then applying $\varphi_j$ to both side of $[Y_{m+\frac{1}{2}},Y_{n+\frac{1}{2}}]=(n-m)I_{m+n+1},$ and using (\ref{der1}) and (\ref{der4}), we obtain
\begin{equation}\label{der2}
(n-m-j)a_{j,m}+(n+j-m)a_{j,n}=0.
\end{equation}
\noindent
Setting $n=0$ in (\ref{der2}), we get $(-m-j)a_{j,m}+(j-m)a_{j,0}=0.$ Then taking $m=-j$ in this equation, we derive $2ja_{j,0}=0,$ which implies $a_{j,0}=0.$ Consequently, we  have  $a_{j,m}=0$ for $m\neq-j.$ Letting $m=j$ and $n=-j$ in (\ref{der2}), we obtain $a_{j,-j}=0$, this shows   $a_{j,m}=0$  for all $m\in\mathbb{Z}.$
This proves $\varphi_{j}=0$ for $j\in{\mathbb{Z}\setminus\{0\}}. $

  \item[(2)] If $j=0,$ then applying $\varphi_0$ to both side of 
$[Y_{m+s},Y_{n+s}]= (n-m)I_{m+n+2s}$
and using (\ref{der1}) and (\ref{der4}), we get
\begin{equation*}
(n-m)(a_{0,m}+a_{0,n})=2\lambda(n-m).
\end{equation*}
Letting $m=0$ in this equation, we obtain $n(a_{0,0}+a_{0,n})=2\lambda n,$ and $a_{0,n}=2\lambda-a_{0,0}$ for $n\neq 0.$
It is following that   $a_{0,n}=\lambda$ for all $n\in\mathbb{Z}.$ Hence $\varphi_0=\lambda \ {\rm Id}.$

\end{enumerate}
\end{proof}

\begin{lemma}\label{lemm10}
If $b\neq-1,$ then $\Delta_{j+\frac{1}{2}}({\mathcal W}(a,b,\frac{1}{2}))=0.$  
\end{lemma}

\begin{proof}
Let $b\neq-1$ and $\varphi_{j+\frac{1}{2}}\in\Delta_{j+\frac{1}{2}}({\mathcal W}(a,b,\frac{1}{2}))$  be a homogeneous $\frac{1}{2}$-derivation. Then we have
\begin{equation}\label{der3}
\varphi_{j+\frac{1}{2}}(M_n)\subseteq M_{n+j+\frac{1}{2}}, \quad \varphi_{j+\frac{1}{2}} (M_{n+\frac{1}{2}})\subseteq M_{n+j+1}.
\end{equation}
\noindent
By (\ref{der3}) we can assume that
\begin{center}$\varphi_{j+\frac{1}{2}}(L_n)=\alpha_{j,n}Y_{n+j+\frac{1}{2}}, \quad \varphi_{j+\frac{1}{2}}(I_n)=\beta_{j,n}Y_{n+j+\frac{1}{2}}, \quad \varphi_{j+\frac{1}{2}}(Y_{n+\frac{1}{2}})=\gamma_{j,n}L_{n+j+1}+\mu_{j,n}I_{n+j+1}.$\end{center}

Let us say $\varphi=\varphi_{j+\frac{1}{2}}.$
The algebra ${\mathcal W}(a,b,\frac{1}{2})$ admits a
$\mathbb{Z}_2$-grading.
Namely, 
\begin{center}${\mathcal W}(a,b,\frac{1}{2})_{\overline{0}}=\langle M_j \ | \ j \in \mathbb Z\rangle$
and 
${\mathcal W}(a,b,\frac{1}{2})_{\overline{1}}=\langle M_{j+\frac{1}{2}} \ | \ j \in \mathbb Z\rangle.$\end{center}

The mapping $\varphi$ changes the grading components. 
It is known, that the commutator of one derivation and one $\frac{1}{2}$-derivation gives a new $\frac{1}{2}$-derivation.
Hence, 
$[\varphi, {\rm ad}_{Y_{n+\frac{1}{2}}}]$ is a $\frac{1}{2}$-derivation which preserve the grading components.
Namely, it is a $\frac 12$-derivation, described in Lemma  \ref{12even}, 
i.e., $[\varphi, {\rm ad}_{Y_{n+\frac{1}{2}}}]= \alpha_n {\rm Id}.$

It is easy to see, that 
\begin{longtable}{lcl}
$\alpha_n I_m $&$=$&$
[\varphi, {\rm ad}_{Y_{n+\frac{1}{2}}}](I_m)=
\varphi[I_m, Y_{n+\frac{1}{2}}]-[\varphi(I_m), Y_{n+\frac{1}{2}}]=$\\
& $=$ &$-\beta_{j,m}[Y_{m+j+\frac{1}{2}}, Y_{n+\frac{1}{2}}]=
-(n-m-j)\beta_{j,m} I_{m+n+j+1}.$
\end{longtable}
For fixed elements $m$ and $j$, we 
can choose an element $n,$ such that $ n\neq m+j$ and $n\neq -1-j.$
The last observations give $\beta_{j,m}=0$ for each $(j,m) \in \mathbb Z \times \mathbb Z.$
Hence,  $\varphi(I_{n+\frac{1}{2}})=0$ and $\alpha_n=0$ for all $n \in \mathbb Z.$

Let us now consider the following observation
\begin{longtable}{lcl}
 $0$ & $=$ &
$[\varphi, {\rm ad}_{Y_{n+\frac{1}{2}}}](L_m)=
\varphi[L_m, Y_{n+\frac{1}{2}}]-[\varphi(L_m), Y_{n+\frac{1}{2}}]=$ \\

& $=$ & 
$(n+\frac{(b-1)m+a+1}{2}) \varphi(Y_{n+m+\frac{1}{2}})
-\alpha_{j,m}[Y_{m+j+\frac{1}{2}},Y_{n+\frac{1}{2}}]=$\\

& $=$ & 
$(n+\frac{(b-1)m+a+1}{2})(\gamma_{j,n+m}L_{n+m+j+\frac{1}{2}}+\mu_{j,n+m}I_{n+m+j+1})- \alpha_{j,m}(n-m-j)I_{m+j+n+1},$
\end{longtable}
which gives two relations considered below.
\begin{enumerate}

    \item[First,] 
$(2n+(b-1)m +a+1)\gamma_{j,m+n}=0.$
Taking $m=0$ and $n=0,$ we obtain
\begin{center}
$((b-1)n+a+1)\gamma_{j,n}=0$ and 
$(2n+a+1)\gamma_{j,n}=0.$
\end{center}
If $b\neq3$ or $a\neq -1-2n,$ we have $\gamma_{j,n}=0.$
In the opposite case, if  $b=3$ and $a= -1-2n,$ we consider 
\begin{longtable}{lcl}
$0$&$=$&$2(n-k)\varphi(I_{k+n+1})=
2\varphi[Y_{k+\frac 12}, Y_{n+\frac 12}]=$\\
&$=$ &$
\mu_{j,k}[I_{k+j+\frac 12}, Y_{n+\frac 12}]+\gamma_{j,n}[Y_{k+\frac 12}, L_{n+j+1}]+\mu_{j,n}[Y_{k+\frac 12}, I_{n+j+1}]=$\\
&$=$&$-\gamma_{j,n}\Big(k+\frac{(b-1)(n+j+1)+a+1}{2}\Big)L_{n+k+j+\frac{3}{2}}.$
\end{longtable}

Taking a suitable $k$, we obtain that $\gamma_{j,n}=0$ 
and $\varphi(Y_{n+\frac{1}{2}})=\mu_{j,n}I_{n+j+1}.$

\item[Second,] 
$\big(n+\frac{(b-1)m+a+1}{2}\big)\mu_{j,m+n}=(n-m-j)\alpha_{j,m},$
which gives 
$\frac{2n+a+1}{2}\mu_{j,n}=\alpha_{j,0} (n-j).$
Let us take an element $n,$ such that $n\neq m+j$ and $2(m+n)+a+1\neq0,$ hence

\begin{longtable}{lclcl}
$\alpha_{j,m}$&$=
$&$\frac{2n+(b-1)m+a+1}{n-m-j}  \mu_{j,n}$&$=$&$
\frac{(2n+(b-1)m+a+1)(n+m-j)}{(n-m-j)(2(n+m)+a+1)}\alpha_{j,0}.$
\end{longtable}
The last observation, due to almost arbitrary $n$, gives that 
$\alpha_{j,m}=0$ for arbitrary $m \in \mathbb Z,$
which gives $\mu_{j,m}=0$ for arbitrary  $m \in \mathbb Z.$
\end{enumerate}

Summarizing, $\varphi=0.$
 
\end{proof}

\begin{theorem}\label{main2}
Let $\varphi$ be a $\frac{1}{2}$-derivation of the algebra ${\mathcal W}(a,-1,\frac{1}{2}),$ then there are three  sets of elements from the basic field $\{\alpha_t\}_{t\in\mathbb{Z}},$ $\{\beta_t\}_{t\in\mathbb{Z}}$ and $\{\gamma_t\}_{t\in\mathbb{Z}}$
such that
\begin{longtable}{llll}
$\varphi(L_m)$ & = &$\sum\limits_{t\in{\mathbb{Z}}}\alpha_tL_{m+t}+\sum\limits_{t\in{\mathbb{Z}}}\beta_tI_{m+t}+\sum\limits_{t\in{\mathbb{Z}}}\gamma_tY_{m+t+\frac{1}{2}},$\\
$\varphi(I_m)$ & = &$\sum\limits_{t\in{\mathbb{Z}}}\alpha_tI_{m+t},$ \\
$\varphi(Y_{m+\frac{1}{2}})$ & = &$\sum\limits_{t\in{\mathbb{Z}}}\alpha_tY_{m+t+\frac{1}{2}}+\sum\limits_{t\in{\mathbb{Z}}}\gamma_tI_{m+t+1}.$
\end{longtable}
\end{theorem}

\begin{proof}
Suppose that $\varphi_j\in \Delta_j({\mathcal W}(a,-1,\frac{1}{2}))$ for     $j\in \mathbb{Z}.$
 By   Theorem \ref{dif}, we have
\begin{equation}\label{der6}
\varphi_j(L_n)=\alpha_j L_{n+j}+\beta_j I_{n+j}, \quad \varphi_j(I_n)=\alpha_{j} I_{n+j}, \quad \varphi_j(Y_{n+\frac{1}{2}})=a_{j,n} Y_{n+j+\frac{1}{2}}.
\end{equation}

Applying $\varphi_j$ to both side of $[Y_{m+\frac{1}{2}},Y_{n+\frac{1}{2}}]=(n-m)I_{m+n+1}$
and using (\ref{der6}), we obtain
\begin{equation}
(n-m-j)a_{j,m}+(n-m+j)a_{j,n}=2(n-m)\alpha_j. 
\label{eq_ii}\end{equation}

\begin{enumerate}
    \item Taking $m=0$ and $n=j$ in (\ref{eq_ii}), we have $ja_{j,j}=j\alpha_j,$ so $a_{j,j}=\alpha_j$ for $j\neq0.$

 \item Taking  $m=j$ in (\ref{eq_ii}), we have   $na_{j,n}=n\alpha_j,$ this implies $a_{j,n}=\alpha_j$ for $n\neq0.$
 
\item Taking  $m=0$ and $n\notin \{0, j\}$ in (\ref{eq_ii}), we have  $a_{j,0}=\alpha_j$ for $j\neq 0.$
\item Taking $j=0$ and   $m=0$ in (\ref{eq_ii}),
we have $a_{0,n}=2\alpha_0-a_{0,0}$ for any $n\neq0$ and choosing  $m\neq n,$ we get  $a_{0,0}=\alpha_0$ and thus $a_{0,n}=\alpha_0$
for all $n\in \mathbb{Z}.$  

\end{enumerate}

As a conclusion, we obtain that $$\varphi_{j}(L_n)=\alpha_{j}L_{n+j}+\beta_jI_{n+j}, \quad \varphi_{j}(I_n)=\alpha_{j}I_{n+j}, \quad \varphi_{j}(Y_{n+\frac{1}{2}})=\alpha_{j}Y_{n+j+\frac{1}{2}}.$$

\medskip 

Let $\varphi_{j+\frac{1}{2}}\in\Delta_{j+\frac{1}{2}}({\mathcal W}(a,-1,\frac{1}{2}))$ for   $j\in \mathbb{Z}.$
Then, we can assume
$$\varphi_{j+\frac{1}{2}}(L_n)=\alpha_{j,n}Y_{n+j+\frac{1}{2}}, \quad \varphi_{j+\frac{1}{2}}(I_n)=\beta_{j,n}Y_{n+j+\frac{1}{2}}, \quad
\varphi_{j+\frac{1}{2}}(Y_{n+\frac{1}{2}})=\gamma_{j,n}L_{n+j+1}+\mu_{j,n}I_{n+j+1} .$$

Applying
$\varphi_{j+\frac{1}{2}}$ to both sides of the first three  relations from the multiplication table, we obtain 
\begin{eqnarray}
(2(n+j)+1-2m+a)\alpha_{j,n}-(2(m+j)+1-2n+a)\alpha_{j,m}=4(n-m)\alpha_{j,m+n},\label{2.1}\end{eqnarray}
\begin{eqnarray}
(2(n+j)+1-2m+a)\beta_{j,n}=4(n-m+a)\beta_{j,m+n}, \label{2.2} \end{eqnarray}
\begin{eqnarray}
(n+j+1-m)\gamma_{j,n}=(2n+1-2m+a)\gamma_{j,m+n}, \label{2.3}\end{eqnarray}
\begin{eqnarray}(n-m-j)\alpha_{j,m}+(n+j+1-m+a)\mu_{j,n}=(2n+1-2m+a)\mu_{j,m+n}.\label{2.4}\end{eqnarray}

\begin{enumerate}
    \item[Case 1.] Now we assume that $j\neq0.$
 Taking $m=0$ in (\ref{2.2}), we have that 
 $(1+2j-2n-3a)\beta_{j,n}=0.$ Hence, $\beta_{j,n}=0$ if $1+2j-2n\neq3a$.
 If $1+2j-2n=3a$ and $m\neq 0,$ then $\beta_{j,m+n}=0$ and 
 $(2(n+j)+1-2m+a)\beta_{j,n}=0.$ The last gives $\beta_{j,n}=0$ for all $n \in \mathbb{Z} $.  
 Similary, from (\ref{2.3}),  we have   $\gamma_{j,m}=0$ for all $m.$
Then by (\ref{2.1}),
we have $\alpha_{j,n}=\alpha_{j,0}$ for all $n\in \mathbb{Z}$.

\begin{itemize}
    \item [A.] Let $a\notin 2\mathbb{Z}+1.$ 
    \begin{enumerate}
        \item 
Taking $m=n=0$ in (\ref{2.4}), we have  $j\alpha_{j,0}=j\mu_{j,0},$ so this implies $\mu_{j,0}=\alpha_{j,0}$ for $j\neq0.$

\item Taking  $n=0$ in (\ref{2.4}), we have   $ (1-2m+a)(\alpha_{j,0}-\mu_{j,m})=0.$ 
 Notice that $a\neq 2m-1,$ hence $\mu_{j,n}=\alpha_{j,0}$ for all  $n\in \mathbb{Z}.$

    \end{enumerate}
\item [B.] Let  $a\in2\mathbb{Z}+1\setminus\{-1\}$. 
\begin{enumerate}

\item Taking $m=0$ in (\ref{2.4}), we have  $(n-j)\alpha_{j,0}=(n-j)\mu_{j,n}.$ 
So this implies $\mu_{j,n}=\alpha_{j,0}$ for $n\neq j.$
\item Taking $m=2j,$  $n=j$ in (\ref{2.4}), we get $(a+1)\mu_{j,j}=(a+1)\alpha_{j,0}.$  It follows $\mu_{j,n}=\alpha_{j,0}$ for all  $n\in \mathbb{Z}.$

\end{enumerate}
\item [C.] Let $a=-1$.
\begin{enumerate}
\item 
Taking  $m=0$ in (\ref{2.4}), we have   $(n-j)\alpha_{j,0}=(n-j)\mu_{j,n}.$ 
So this implies $\mu_{j,n}=\alpha_{j,0}$ for $n\neq j.$
\item Taking $m=2j$ and  $n=-j$ in (\ref{2.4}), we have  $\mu_{j,j}= \alpha_{j,0}.$  It follows $\mu_{j,n}=\alpha_{j,0}$ for all  $n\in \mathbb{Z}.$

\end{enumerate}
\end{itemize}

\medskip 

 \item[Case 2.] Consider the case of $j=0.$
It is easy to see, that (\ref{2.1})
gives $\alpha_{0,n}=\alpha_{0,0}$ for all $n\in \mathbb{Z}.$

\begin{itemize}
    \item [A.] If $a\notin 2\mathbb{Z}+1.$
Taking  $m=0$ in (\ref{2.4}), we have  $n\alpha_{0,0}=n\mu_{0,n},$ so this implies $\mu_{0,n}=\alpha_{0,0}$ for $n\neq0.$
 Taking  $m=2$ and $n=0$ in (\ref{2.4}), we have   
 $(a-1)(\alpha_{0,0}-\mu_{0,0})=0.$  It follows $\mu_{0,n}=\alpha_{0,0}$ for all  $n\in \mathbb{Z}.$

\item [B.] If $a\in2\mathbb{Z}+1\setminus\{-1\}.$ 
Taking $m=0$ in (\ref{2.4}), we have  $n\alpha_{0,0}=n\mu_{0,n},$ so this implies $\mu_{0,n}=\alpha_{0,0}$ for $n\neq 0.$
Taking $m=1$ and   $n=0$ in (\ref{2.4}), we have   $a\mu_{0,0}=a\alpha_{0,0}.$  It follows $\mu_{0,n}=\alpha_{0,0}$ for all  $n\in \mathbb{Z}.$

\item [C.] If $a=-1.$ 
Taking 
 $m=0$ in (\ref{2.4}), we have  $n(\mu_{0,n}-\alpha_{0,0})=0$ and this implies $\mu_{0,n}=\alpha_{0,0}$
for $n\neq 0.$ Taking  $m=1$ and $n=0$
in (\ref{2.4}), we have  $\mu_{0,0}=\alpha_{0,0}.$
\end{itemize}

\end{enumerate}

Summarizing, the Theorem is proved.
\end{proof}

Now we consider the algebra ${\mathcal W}$
with a basis $\{ L_i, I_i,Y_{i+\frac{1}{2}}\ | \ {i \in \mathbb Z}\}$ given by the following 
commutative multiplication table:
\begin{longtable}{lcllcllcllcl} 
$L_iL_j$&$=$&$L_{i+j},$ & $  L_iI_j$&$=$&$I_{i+j},$ & 
$L_iY_{j+\frac{1}{2}}$&$=$&$Y_{i+j+\frac{1}{2}},$ & 
$Y_{i+\frac{1}{2}}Y_{j+\frac{1}{2}}$&$=$&$I_{i+j+1}.$ 
\end{longtable} 

In the following, we aim to classify all transposed Poisson structures on  ${\mathcal W}(a,-1,\frac{1}{2}).$
\begin{theorem} Let $(\mathfrak {L},\cdot,[\cdot,\cdot])$ be a transposed Poisson structure defined on the Lie algebra ${\mathcal W}(a,-1,\frac{1}{2}).$
Then $(\mathfrak {L},\cdot,[\cdot,\cdot])$ is not Poisson algebra and $(\mathfrak {L},\cdot)$ is a mutation of the algebra $\mathcal W$. On the other hand, every mutation of the algebra $\mathcal W$ gives a transposed Poisson structure with the Lie part isomorphic to ${\mathcal W}(a,-1,\frac{1}{2}).$

\end{theorem}
\begin{proof} We aim to describe the multiplication $\cdot.$ By Lemma \ref{l1}, for every element  $X\in\{L_i, I_i, Y_{i+\frac{1}{2}}\ | \  {i\in\mathbb Z}\},$ there is a related
$\frac{1}{2}$-derivation $\varphi_{X}$ of ${\mathcal W}(a,-1,\frac{1}{2})$ such that $\varphi_{X}(Y)= X \cdot Y.$
Then by Theorem \ref{main2}, we have that 
\begin{longtable}{llll}
$\varphi_{X}(L_m)$ & = &$\sum\limits_{t\in{\mathbb{Z}}}\alpha_{t, X}L_{m+t}+\sum\limits_{t\in{\mathbb{Z}}}\beta_{t, X}I_{m+t}+\sum\limits_{t\in{\mathbb{Z}}}\gamma_{t, X}Y_{m+t+\frac{1}{2}},$\\
$\varphi_{X}(I_m)$ & = &$\sum\limits_{t\in{\mathbb{Z}}}\alpha_{t,X}I_{m+t},$ \\
$\varphi_{X}(Y_{m+\frac{1}{2}})$ & = &$\sum\limits_{t\in{\mathbb{Z}}}\alpha_{t,X}Y_{m+t+\frac{1}{2}}+\sum\limits_{t\in{\mathbb{Z}}}\gamma_{t,X}I_{m+t+1}.$
\end{longtable}

Now we consider $\varphi_{X}(Y) = X \cdot Y = Y \cdot X = \varphi_{Y}(X)$ for $X,Y \in   \{ L_i, I_i,Y_{i+\frac{1}{2}}\ | \  {i \in \mathbb Z} \}.$

\begin{enumerate}[I.]
    \item 
Let 
$X=I_m$ and $Y=Y_{n+\frac{1}{2}}.$ Then, from $\varphi_{I_m}(Y_{n+\frac{1}{2}})=\varphi_{Y_{n+\frac{1}{2}}}(I_m),$ we have 
$$\sum\limits_{t\in \mathbb{Z}}\alpha_{t,I_m}Y_{n+t+\frac{1}{2}}+\sum\limits_{t\in \mathbb{Z}}\gamma_{t,I_m}I_{n+t+1}=\sum\limits_{t\in \mathbb{Z}}\alpha_{t,Y_{n+\frac{1}{2}}}I_{m+t}.$$

The last equality gives $\alpha_{t,I_m}=0$  for all $m,t\in\mathbb{Z}.$

\item 
Let $X=I_m$ and $Y= I_n,$ then we have
$$I_m\cdot I_n=\varphi_{I_m}(I_n)=\sum\limits_{t\in{\mathbb{Z}}}\alpha_{t,I_m}I_{n+t}=0.$$

\item 
Let $X=I_m$ and $Y= L_n.$ Then from $\varphi_{I_m}(L_n)=\varphi_{L_n}(I_m),$ we get
$$\sum\limits_{t\in \mathbb{Z}}\beta_{t,I_m}I_{n+t}+\sum\limits_{t\in \mathbb{Z}}\gamma_{t,I_m}Y_{n+t+\frac{1}{2}}=\sum\limits_{t\in \mathbb{Z}}\alpha_{t,L_n}I_{m+t},$$
which implies $\gamma_{t,I_m}=0.$ Hence, 
$I_m\cdot Y_{n+\frac{1}{2}}=0$   for all $m,n\in\mathbb{Z}.$


\item Let $X=L_m$ and $Y= L_0.$ Then the equality $\varphi_{L_m}(L_0)=L_m\cdot L_0=L_0\cdot L_m=\varphi_{L_0}(L_m),$ gives
$$\sum\limits_{t\in \mathbb{Z}}\alpha_{t,L_m}L_{t}+\sum\limits_{t\in \mathbb{Z}}\beta_{t,L_m}I_{t}+\sum\limits_{t\in \mathbb{Z}}\gamma_{t,L_m}Y_{t+\frac{1}{2}}=\sum\limits_{t\in \mathbb{Z}}\alpha_{t,L_0}L_{m+t}+\sum\limits_{t\in \mathbb{Z}}\beta_{t,L_0}I_{m+t}+\sum\limits_{t\in \mathbb{Z}}\gamma_{t,L_0}Y_{m+t+\frac{1}{2}}.$$

Hence, we obtain  $\alpha_{k,L_m}=\alpha_{k-m,L_0},$ $\beta_{k,L_m}=\beta_{k-m,L_0}$ and $\gamma_{k,L_m}=\gamma_{k-m,L_0}.$


\item Let $X=L_0$ and $Y= Y_{n+\frac{1}{2}},$ then from 
$\varphi_{L_0}(Y_{n+\frac{1}{2}})=\varphi_{Y_{n+\frac{1}{2}}}(L_0),$ we get
$$\sum\limits_{t\in \mathbb{Z}}\alpha_{t,L_0}Y_{n+t+\frac{1}{2}}+\sum_{t\in \mathbb{Z}}\gamma_{t,L_0}I_{n+t+1}=\sum\limits_{t\in \mathbb{Z}}\beta_{t,Y_{n+\frac{1}{2}}}I_{t}+\sum\limits_{t\in \mathbb{Z}}\gamma_{t,Y_{n+\frac{1}{2}}}Y_{t+\frac{1}{2}}.$$

Thus, we obtain  $\gamma_{k,Y_{n+\frac{1}{2}}}=\alpha_{k-n, L_0},$ $\beta_{k,Y_{n+\frac{1}{2}}}=\gamma_{k-n,L_0}.$ 
\end{enumerate}

Summarizing all the above parts, we have that the multiplication table of $(\mathfrak, \cdot)$ is given by following non-trivial relations. 
\begin{longtable}{lcl}
$L_m\cdot L_n $&$=$&$\sum\limits_{t\in \mathbb{Z}}\alpha_{t,L_0}L_{m+n+t}+\sum\limits_{t\in \mathbb{Z}}\beta_{t,L_0}I_{m+n+t}+\sum\limits_{t\in \mathbb{Z}}\gamma_{t,L_0}Y_{m+n+t+\frac{1}{2}},$\\
$L_m\cdot I_n$&$=$&$\sum\limits_{t\in \mathbb{Z}}\alpha_{t,L_0}I_{m+n+t},$\\
$L_m\cdot Y_{n+\frac{1}{2}}$&$=$&$\sum\limits_{t\in \mathbb{Z}}\alpha_{t,L_0}Y_{m+n+t+\frac{1}{2}}+\sum\limits_{t\in \mathbb{Z}}\gamma_{t,L_0}I_{m+n+t+1},$\\
$Y_{m+\frac{1}{2}}\cdot Y_{n+\frac{1}{2}}$&$=$&$\sum\limits_{t\in \mathbb{Z}}\alpha_{t,L_0}I_{m+n+t+1}.$
\end{longtable}

Let us define $w:=\sum\limits_{t\in \mathbb{Z}}\alpha_{t,L_0}L_{t}+\sum\limits_{t\in \mathbb{Z}}\beta_{t,L_0}I_{t}+\sum\limits_{t\in \mathbb{Z}}\gamma_{t,L_0}Y_{t+\frac{1}{2}}.$

Then all nonzero multiplications on the basis elements are given below. 
\begin{longtable}{lcllcl}
$L_m\cdot L_n $&$= $&$L_mwL_n, $&$ L_m\cdot I_n$&$=$&$L_mwI_n,$\\
$L_m\cdot Y_{n+\frac{1}{2}}$&$=$&$L_mwY_{n+\frac{1}{2}}, $&$ Y_{m+\frac{1}{2}}\cdot Y_{n+\frac{1}{2}}$&$=$&$Y_{m+\frac{1}{2}}wY_{n+\frac{1}{2}},$\\
\end{longtable}
Analyzing the associative law of $\cdot,$ we can see that the multiplication $\cdot$ is associative.
The algebra $(\mathfrak{L},\cdot)$ is a mutation of $\mathcal  W$. On the other hand,  for any element $w$ the product $\cdot$
gives a transposed Poisson structure defined on the Lie algebra ${\mathcal W}(a,-1,\frac{1}{2})$ for an arbitrary complex number $a.$
It gives the complete statement of the theorem.
\end{proof}

\section{Transposed Poisson   structures on 
 not-finitely graded Witt algebra $W_n(G)$} \label{wnG}

Let $n \in \mathbb Z$ and 
$G$ be any nontrivial additive subgroup of $\mathbb C.$ The algebra
$W_n(G)$ has a basis $\{L_{\alpha, i}  \ | \ \alpha \in G, i \in \mathbb{Z}\}$ and the multiplication table is given below: 
 \begin{equation}\label{ab11}
     [L_{\alpha, i}, L_{\beta, j}] = (\beta-\alpha)L_{\alpha+\beta, i+j} + (j-i) L_{\alpha+\beta, i+j+n},
 \end{equation}
for $\alpha, \beta\in G,$ $i,j\in\mathbb{Z}.$

\begin{theorem} \label{main5}
Let $\varphi$ be a $\frac{1}{2}$-derivation of the algebra $W_n(G),$ then 
there is a finite set $\{ a^{d,m} \}_{d \in G, m \in \mathbb Z}$ 
of complex elements, such that 
$$\varphi(L_{\alpha,i})={ \sum_{d\in G, \ m\in\mathbb{Z}}}a^{d,m}L_{\alpha+d,i+m}.$$
\end{theorem}

\begin{proof}
The algebra $W_0(G)$ is a generalized Witt algebra considered in \cite{kk23}.
Thanks to \cite[Proposition 16]{kk23}, the Theorem is proved for $n=0.$
Below, we will consider only the case of $n\neq 0.$

The   algebra  $W_n(G)$ is $G$-graded
\begin{center}
    $W_n(G)=\bigoplus\limits_{\alpha\in G}W_n(G)_{\alpha},$ where $ W_n(G)_{\alpha}=\langle L_{\alpha,i} \ | \ i\in \mathbb{Z}\rangle$ for $\alpha \in G.$
\end{center}
 By Lemma \ref{l01} we get 
$\triangle(W_n(G))=\bigoplus\limits_{\alpha\in G}\triangle_{\alpha}(W_n(G)).$

Let $\varphi_d \in \Delta_{d}(W_n(G)).$   We can assume that 
$\varphi_d(L_{\alpha,i})=\sum\limits_{k\in\mathbb{Z}}a_{\alpha,i}^{d,k}L_{\alpha+d,k}.$ To determine these coefficients, we start with applying $\varphi_d$
to both side of (\ref{ab11}) and obtain
\begin{equation}\label{maini}
\begin{array}{cl}
2\sum\limits_{k\in\mathbb{Z}} & \left((\beta-\alpha)a_{\alpha+\beta,i+j}^{d,k}+(j-i)a_{\alpha+\beta,i+j+n}^{d,k}\right)  L_{\alpha+\beta+d,k}=\\[2mm]
&\multicolumn{1}{r}{=\sum\limits_{k\in\mathbb{Z}}a_{\alpha,i}^{d,k}\left((\beta-\alpha-d)L_{\alpha+\beta+d,j+k}+(j-k)L_{\alpha+\beta+d,j+k+n}\right)+}\\[2mm]
&\multicolumn{1}{r}{+\sum\limits_{k\in\mathbb{Z}}a_{\beta,j}^{d,k}\left((\beta-\alpha+d)L_{\alpha+\beta+d,i+k}+(k-i)L_{\alpha+\beta+d,i+k+n}\right).}
\end{array}
\end{equation}

\noindent
Letting $\beta=i=j=0$  in (\ref{maini}) for $k\in\mathbb{Z} ,$ we get  
\begin{equation} \label{k1}
(d-\alpha) a_{\alpha,0}^{d,k}+(k-n)a_{\alpha,0}^{d,k-n}=(d-\alpha) a_{0,0}^{d,k}+(k-n)a_{0,0}^{d,k-n}.
\end{equation}
Putting $\beta=-\alpha,$ $i=j=0$ in (\ref{maini}) with (\ref{k1}) for $k\in \mathbb{Z}$, we  get 
\begin{equation}\label{k2}
  2a_{0,0}^{d,k}= a_{\alpha,0}^{d,k}+a_{-\alpha,0}^{d,k}. 
\end{equation}
In the equation (\ref{k1}) replacing $\alpha$ with $-\alpha,$ we obtain:
\begin{equation} \label{k3}
(d+\alpha) a_{-\alpha,0}^{d,k}+(k-n)a_{-\alpha,0}^{d,k-n}=(d+\alpha) a_{0,0}^{d,k}+(k-n)a_{0,0}^{d,k-n}
\end{equation}
and adding the equations (\ref{k1}) and (\ref{k3}) with together (\ref{k2}) we deduce 
$$a_{\alpha,0}^{d,k}=a_{0,0}^{d,k} \quad \text{for all} \quad \alpha \in G.$$

\noindent
Putting $\beta=-\alpha,$ $j=0$ in (\ref{maini}), we  get 
\begin{equation}\label{k4}
  (d+2\alpha)a_{\alpha,i}^{d,k}+(k-n)a_{\alpha,i}^{d,k-n}=4\alpha a_{0,i}^{d,k}+2ia_{0,i+n}^{d,k}+(d-2\alpha)a_{0,0}^{d,k-i}+(k-2i-n)a_{0,0}^{d,k-i-n}. 
\end{equation}

\noindent
Letting $\beta=j=0$ and  in (\ref{maini}), we deduce
\begin{equation}\label{k5}
  (d-\alpha)a_{\alpha,i}^{d,k}+(k-n)a_{\alpha,i}^{d,k-n}=2ia_{\alpha,i+n}^{d,k}+(d-\alpha)a_{0,0}^{d,k-i}+(k-2i-n)a_{0,0}^{d,k-i-n}. 
\end{equation}

\noindent
Subtracting (\ref{k4}) from (\ref{k5}) it gives 
\begin{equation}\label{k6}
\alpha(3a_{\alpha,i}^{d,k}+a_{0,0}^{d,k-i}-4a_{0,i}^{d,k})=2i(a_{0,i+n}^{d,k}-a_{\alpha,i+n}^{d,k}). 
\end{equation}
\noindent
Now consider the case of  $\beta=i=0$ in (\ref{maini}), we have
\begin{equation}\label{k7}
  (d-\alpha)a_{0,j}^{d,k}+(k-n)a_{0,j}^{d,k-n}=2ja_{\alpha,j+n}^{d,k}-2\alpha a_{\alpha,j}^{d,k} +(d+\alpha)a_{0,0}^{d,k-j}+(k-2j-n)a_{0,0}^{d,k-j-n}. 
\end{equation}
Setting $\alpha=0$ in (\ref{k5}) and substituting  $j$ with $i$ in (\ref{k7}) then combining them it derives 
\begin{equation}\label{k8}
\alpha(2a_{\alpha,i}^{d,k}-a_{0,0}^{d,k-i}-a_{0,i}^{d,k})=2i(a_{\alpha,i+n}^{d,k}-a_{0,i+n}^{d,k}).
\end{equation}

\noindent
Finally, adding up (\ref{k6}) and (\ref{k8}),  we have  $a_{\alpha,i}^{d,k}=a_{0,i}^{d,k},$
and from  the equation (\ref{k6}) it follows that
$a_{\alpha,i}^{d,k}=a_{0,i}^{d,k}=a_{0,0}^{d,k-i}.$
Thus, it proves that $\varphi_d$ has the form 
$$\varphi_d(L_{\alpha,i})=\sum_{k\in\mathbb{Z}}a_{0,0}^{d,k-i}L_{\alpha+d,k}.$$
\end{proof}

Now, we aim to classify all transposed Poisson structures on $W_n(G).$
Let $\mathcal{W}$ be a commutative associative algebra with basis $\{L_{\alpha, i} \ | \ \alpha \in G, i \in \mathbb{Z}\}$ and multiplication $$L_{\alpha, i} L_{\beta, j} = L_{\alpha+\beta, i+j}.$$

\begin{theorem}\label{tpwn} Let $(\mathfrak {L},\cdot,[\cdot,\cdot])$ be a transposed Poisson structure defined on the Lie algebra $W_n(G).$ Then $(\mathfrak {L},\cdot,[\cdot,\cdot])$ is not Poisson algebra and $(\mathfrak {L},\cdot)$ is a mutation of the algebra $\mathcal{W}$. On the other hand, 
every mutation of the algebra $\mathcal{W}$ gives a transposed Poisson structure with the Lie part isomorphic to the algebra $W_n(G).$
\end{theorem}
\begin{proof} We aim to describe the multiplication $\cdot.$ By Lemma \ref{l1}, for every element  $X\in \{L_{\alpha, i} \ | \ \alpha \in G, i \in \mathbb{Z}\},$ there is a related
$\frac{1}{2}$-derivation $\varphi_{l}$ of $W_n(G)$ such that 
$\varphi_{X}(Y)=X\cdot Y.$ By Theorem \ref{main5}, for each $\varphi \in \Delta(W_n(G)),$ we have
$$\varphi(L_{\alpha,i})={ \sum_{d\in G, \ k\in\mathbb{Z}}}a^{d,k}L_{\alpha+d,i+k}.$$

\noindent
Now we consider  the multiplication $L_{0,0}\cdot L_{\alpha,i}$ for all $i\in \mathbb{Z}\setminus\{0\}.$ Then we have
$$\varphi_{L_{0,0}}(L_{\alpha,i})\ =\ L_{\alpha,i} \cdot L_{0,0}= L_{0,0}\cdot L_{\alpha,i}\ =\ \varphi_{L_{\alpha,i}}(L_{0,0}),$$
which implies
$$\sum\limits_{d\in G}\sum\limits_{k\in\mathbb{Z}}a_{L_{0,0}}^{d,k}L_{\alpha+d,i+k}
=\sum\limits_{d\in G}\sum\limits_{k\in\mathbb{Z}}a_{L_{\alpha,i}}^{d,k}L_{d,k}.$$

\noindent Thus, we get $a_{L_{\alpha,i}}^{d,k}=a_{L_{0,0}}^{d-\alpha,k-i}$.
Now,    we determine the multiplication $L_{\alpha,i}\cdot L_{\beta,j}:$

\begin{center}
    $L_{\alpha,i}\cdot L_{\beta,j}=\sum\limits_{d\in G}\varphi_{d,L_{\alpha,i}}(L_{\beta,j})=
\sum\limits_{d\in G}\sum\limits_{k\in\mathbb{Z}}a_{L_{\alpha,i}}^{d,k}L_{\beta+d,k+j}
=\sum\limits_{d\in G}\sum\limits_{k\in\mathbb{Z}}a_{L_{0,0}}^{d-\alpha,k-i}L_{\beta+d,k+j}.$
\end{center}
Hence, we get 
 $$L_{\alpha,i}\cdot L_{\beta,j}=\sum_{d\in G}\sum_{k\in\mathbb{Z}}a_{L_{0,0}}^{d,k}L_{d+\alpha+\beta,k+i+j},$$

\noindent
If we denote by $w= \sum\limits_{d\in G}\sum\limits_{k\in\mathbb{Z}}a_{L_{0,0}}^{d,k}L_{d,k},$ then we obtain
$L_{\alpha,i}\cdot L_{\beta,j}=L_{\alpha,i}w L_{\beta,j}.$
Analyzing the associative law of $\cdot,$ we can see that the multiplication $\cdot$ is associative. Thus, we get that the algebra $(\mathfrak{L},\cdot)$ is a mutation of $\mathcal{W}$. On the other hand, for any element $w$ the product $\cdot$
gives a transposed Poisson structure defined on the Lie algebra $W_n(G).$  It is easy to see that $\mathfrak{L}\cdot [\mathfrak{L},\mathfrak{L}]\neq 0$ and  $(\mathfrak{L},\cdot,[\cdot,\cdot])$ is a non-Poisson algebra.
It gives the complete statement of the theorem.
\end{proof}

\begin{rem}
It should be mentioned that in \cite{chs15} 
the algebra $W_{1}(G)$ was considered under the restriction $\mathbb Z_+$ instead of 
$\mathbb Z$.
Following this idea, 
let  $G$ be a nontrivial additive subgroups of $\mathbb C$,  and  
$F$ be a nontrivial additive subsemigroup with a neutral element of $\mathbb C$ and $\mathfrak{f} \in F,$
it is possible to define a Lie algebra $W_n(G,F).$
The algebra
$W_n(G,F)$ has a basis $\{L_{g, f}  \ | \ g \in G, f \in F\}$ and the multiplication table is given below: 
 \begin{equation*}\label{ab1}
     [L_{g_1, f_1}, L_{g_2, f_2}] = (g_2-g_1)L_{g_1+g_2, f_1+f_2} + (f_2-f_1) L_{g_1+g_2, f_1+f_2+\mathfrak{f}}.
 \end{equation*}
 The description of $\frac 12$-derivations and transposed Poisson structures on $W_n(G,F)$
 is completely similar to Theorems \ref{main5} and \ref{tpwn}.

\end{rem}

\section{Transposed Poisson  structures on 
 not-finitely graded Heisenberg-Witt algebra $HW_n(G).$
} 

Let $n \in \mathbb Z$ and $G$ be any nontrivial additive subgroup of $\mathbb{C}$.
The Lie algebra  $HW_n(G)$  has a 
  basis $\{L_{\alpha, i}, \   H_{\alpha, i} \ | \ \alpha \in G, \ i \in \mathbb{Z}\}$ and the following multiplication table:  
 \begin{eqnarray}
     &&[L_{\alpha, i}, L_{\beta, j}] = (\beta - \alpha)L_{\alpha+\beta, i+j} + (j-i) L_{\alpha+\beta, i+j+n},\label{ad}\\
 &&[L_{\alpha, i}, H_{\beta, j}] = \beta H_{\alpha+\beta, i+j} + j H_{\alpha+\beta, i+j+n},\label{cd}\\
 &&[H_{\alpha, i}, H_{\beta, j}] = 0.\label{bd}
 \end{eqnarray}

\begin{theorem} \label{main3}
Let $\varphi$ be a $\frac{1}{2}$-derivation of the algebra $HW_n(G).$ 
\begin{itemize}

\item If $n=0$, then  
\begin{longtable}{lcl}
$\varphi(L_{\alpha,i}) $&$ =$&$\sum\limits_{d\in G \cap \mathbb{Z}}a^{d,-d}L_{\alpha+d,i-d},$\\
$\varphi(H_{\alpha,i}) $&$= $&$\sum\limits_{d\in G \cap \mathbb{Z}}a^{d,-d}H_{\alpha+d,i-d}.$
\end{longtable}
\item If $n\neq0,$ then 
\begin{longtable}{lcl}
$\varphi(L_{\alpha,i})$&$=$&$\sum\limits_{d\in G\setminus\{0\}}\sum\limits_{k\in\mathbb{Z}}a^{d,k}L_{\alpha+d,k+i}+a^{0,0}L_{\alpha,i},$\\
$\varphi(H_{\alpha,i})$&$=$&$\sum\limits_{d\in G\setminus\{0\}}\sum\limits_{k\in\mathbb{Z}}a^{d,k}H_{\alpha+d,k+i}+a^{0,0}H_{\alpha,i},$
\end{longtable}
where
  $a^{d,kn}=0$ for $k\in\mathbb{Z}_+\setminus\{0\}. $
If $k=tn+m$, then 
\begin{enumerate}
\item for $t\in\mathbb{Z}_+\setminus\{0\},$ $m=\overline{1,|n|-1},$ then \ 
$a^{d,k} ={(-1)^t{d^{-t}}\prod\limits_{p=1}^{t-1}(pn+m)}a^{d,m},$ 
\item  for $t\in\mathbb{Z}_-\setminus\{0\},$ $m=\overline{0,|n|-1},$ then \
$a^{d,k} ={d^t}\big(\prod\limits_{p=1}^{t}(pn-m)\big)^{-1}a^{d,m}.$
\end{enumerate}
\end{itemize}
\end{theorem}

\begin{proof}
The  algebra  $HW_n(G)$ is $G$-graded
\begin{center}
$HW_n(G)=\bigoplus\limits_{\alpha\in G}(HW_n(G))_{\alpha},  \quad  (HW_n(G))_{\alpha}=\langle L_{\alpha,i}, H_{\alpha,i} \ |\ i\in \mathbb{Z} \rangle\quad \textrm{for} \ \  \alpha \in G.$\end{center}
By Lemma \ref{l01}, we get 
$\triangle(HW_n(G))=\bigoplus\limits_{\alpha\in G}\triangle_{\alpha}(HW_n(G)).$
Suppose that $\varphi_d\in \Delta_d(H W_n(G)).$ 
The subalgebra $\langle L_{\alpha, i}  \ | \ \alpha \in G, \ i \in \mathbb{Z}\rangle$
is isomorphic to the algebra  $W_n(G)$ considered in the  Section \ref{wnG}. Then by Theorem \ref{main5}, we can assume that 
\begin{longtable}{lcl}
$\varphi_d(L_{\alpha,i})$&$=$&$\sum\limits_{k\in\mathbb{Z}}a^{d,k-i}L_{\alpha+d,k}+\sum\limits_{k\in\mathbb{Z}}b_{\alpha,i}^{d,k}H_{\alpha+d,k},$\\
$\varphi_d(H_{\alpha,i})$&$=$&$\sum\limits_{k\in\mathbb{Z}}f_{\alpha,i}^{d,k}L_{\alpha+d,k}+\sum\limits_{k\in\mathbb{Z}}c_{\alpha,i}^{d,k}H_{\alpha+d,k}.$
\end{longtable}

\begin{enumerate}[I.]
    \item 
    To determine the  coefficients $f_{\alpha,i}^{d,k}$, we apply $\varphi_d$
to (\ref{bd}) and obtain
\begin{equation}\label{n1}
 \sum\limits_{k\in\mathbb{Z}}f_{\alpha,i}^{d,k} \left(\beta H_{\alpha+\beta+d,j+k}+jH_{\alpha+\beta+d,j+k+n}\right) =\sum\limits_{k\in\mathbb{Z}}f_{\beta,j}^{d,k} \left(\alpha H_{\alpha+\beta+d,i+k}+iH_{\alpha+\beta+d,i+k+n}\right). 
\end{equation}

\noindent Taking $\beta=i=j=0$ in (\ref{n1}), it gives 
$f_{0,0}^{d,k}=0.$

\noindent Taking $\beta=-\alpha,$ $i=j=0$ in (\ref{n1}), we have 
$f_{-\alpha,0}^{d,k}=-f_{\alpha,0}^{d,k}.$

\noindent Taking $\beta=-\alpha,$ $j=0$ in (\ref{n1}), we have 
\begin{equation}\label{n6}
 \alpha f_{\alpha,i}^{d,k}=-\alpha f_{-\alpha,0}^{d,k-i}-if_{-\alpha,0}^{d,k-i-n}.    
\end{equation}
Next, to determine the coefficients, applying $\varphi_d$
to both side of (\ref{cd}) obtain
{\small \begin{equation}\label{n3}
\begin{array}{rl}
 2\sum\limits_{k\in\mathbb{Z}} &\left(\beta f_{\alpha+\beta,i+j}^{d,k} +j f_{\alpha+\beta,i+j+n}^{d,k}\right)   L_{\alpha+\beta+d,k} \ =\\
 \sum\limits_{k\in\mathbb{Z}}  & f_{\beta,j}^{d,k} \left((\beta-\alpha+d) L_{\alpha+\beta+d,i+k}+(k-i)L_{\alpha+\beta+d,i+k+n}\right).
 \end{array}
\end{equation}
and
\begin{equation}\label{n4}
\begin{array}{ll}
2\sum\limits_{k\in\mathbb{Z}}&\left(\beta c_{\alpha+\beta,i+j}^{d,k}+jc_{\alpha+\beta,i+j+n}^{d,k}\right)H_{\alpha+\beta+d,k} \ =\\[2mm]
& \sum\limits_{k\in\mathbb{Z}}a^{d,k-i} \left(\beta H_{\alpha+\beta+d,j+k}+jH_{\alpha+\beta+d,j+k+n}\right)+ \\ 
[2mm] & \sum\limits_{k\in\mathbb{Z}}c_{\beta,j}^{d,k}\left((\beta+d)H_{\alpha+\beta+d,i+k}+kH_{\alpha+\beta+d,i+k+n}\right).
\end{array}
\end{equation}} 

\noindent Putting $\beta=-\alpha,$ $i=j=0$ in (\ref{n3}), we obtain 
\begin{equation}\label{n2}
(d-2\alpha)f_{-\alpha,0}^{d,k}+(k-n)f_{-\alpha,0}^{d,k-n}=0.    \end{equation}
Replacing $-\alpha$ with $\alpha$ in (\ref{n2}), it gives  
\begin{equation}\label{n5}
(d+2\alpha)f_{\alpha,0}^{d,k}+(k-n)f_{\alpha,0}^{d,k-n}=0.    \end{equation}
Combining (\ref{n2}) and (\ref{n5}), and using $f_{-\alpha,0}^{d,k}=-f_{\alpha,0}^{d,k},$ we can get
$f_{\alpha,0}^{d,k}=0.$
Moreover, from (\ref{n6}), we have $f_{\alpha,i}^{d,k}=0$ for $k\in\mathbb{Z} .$

\item Now we identify  the coefficients $c_{\alpha,i}^{d,k},$ so 
taking $\alpha=j=0$ in (\ref{n4}), we deduce 
\begin{equation}\label{n7}
  2\beta c_{\beta,i}^{d,k}=\beta a^{d,k-i}+(\beta+d)c_{\beta,0}^{d,k-i} +(k-i-n)c_{\beta,0}^{d,k-i-n}. 
\end{equation}

\noindent  Next taking $\alpha=-\beta,$ $j=0$ in (\ref{n4}), we have 
\begin{equation}\label{n8}
  2\beta c_{0,i}^{d,k}=\beta a^{d,k-i}+(\beta+d)c_{\beta,0}^{d,k-i} + (k-i-n)c_{\beta, 0}^{d,k-i-n}. 
\end{equation}

\noindent Combining (\ref{n7}) and (\ref{n8}), we obtain that 
$$c_{\beta,i}^{d,k}=c_{0,i}^{d,k}.$$

\noindent Now putting $\beta=0,$ $i=0$ in (\ref{n7}), we have 
\begin{equation}\label{n9}
   dc_{0,0}^{d,k}+(k-n)c_{0,0}^{d,k-n}=0. 
\end{equation}
\noindent Letting $\beta\neq 0,$ $i=0$ in (\ref{n7}), we get 
\begin{equation}\label{nm7}
 \beta (a^{d,k}-c_{0,0}^{d,k})+d c_{0,0}^{d,k} +(k-n)c_{0,0}^{d,k-n}=0. 
\end{equation}

Thus, we get 
$c_{0,0}^{d,k}=a^{d,k},$ which implies  $c_{0,i}^{d,k}=a^{d,k-i},$ and 
\begin{equation}\label{n10}
   da^{d,k}+(k-n)a^{d,k-n}=0. 
\end{equation}

\begin{enumerate}
    \item[A.]
If $n=0,$ then we get $(d+k)a^{d,k}=0,$ which implies  $a^{d,k}=0$ for $k\neq-d.$

\item[B.] If $n\neq0,$  then taking $d=0,$ we have  $a^{0,k}=0$ for  $k\neq 0.$
Then letting $k=n,$ we get $a^{d,n}=0$ for $d\neq 0,$ which implies   $a^{d,kn}=0$ for $k\in\mathbb{Z}_+\setminus\{0\}. $

Now, letting $k=tn+q$, we obtain 
\begin{enumerate}
    \item  $a^{d,k} ={(-1)^t{d^{-t}}\prod\limits_{p=1}^{t-1}(pn+q)}a^{d,q}$ for  $t\in\mathbb{Z}_+\setminus\{0\},$ $q=\overline{1,|n|-1},$ 
\item  $a^{d,k} ={d^t}\big(\prod\limits_{p=1}^{t}(pn-q)\big)^{-1}a^{d,q}$ for $t\in\mathbb{Z}_-\setminus\{0\},$ $q=\overline{0,|n|-1}.$ 

\end{enumerate}

\end{enumerate}

\item Next to determine the coefficients $b_{\alpha,i}^{d,k},$  we apply $\varphi_d$
to both side of (\ref{ad}) and obtain

\begin{equation}\label{n11}
\begin{array}{ll}
2\sum\limits_{k\in\mathbb{Z}} & \left((\beta-\alpha) b_{\alpha+\beta,i+j}^{d,k}+(j-i)b_{\alpha+\beta,i+j+n}^{d,k}\right)H_{\alpha+\beta+d,k}=\\[2mm]
& \sum\limits_{k\in\mathbb{Z}}b_{\beta,j}^{d,k}\left((\beta+d) H_{\alpha+\beta+d,i+k}+kH_{\alpha+\beta+d,i+k+n}\right)-\\[2mm]
&\sum\limits_{k\in\mathbb{Z}}b_{\alpha,i}^{d,k}\left((\alpha+d)H_{\alpha+\beta+d,j+k}+kH_{\alpha+\beta+d,j+k+n}\right).
\end{array}
\end{equation}

\noindent We start with taking $\beta=i=j=0$ in (\ref{n11}), it gives  
\begin{equation}\label{n12}
  (d-\alpha)b_{\alpha,0}^{d,k}+(k-n)b_{\alpha,0}^{d,k-n}=
  db_{0,0}^{d,k}+(k-n)b_{0,0}^{d,k-n}.
\end{equation} 
\noindent Taking $\beta=-\alpha,$ $i=j=0$ in (\ref{n11}) with  (\ref{n12}) we have 
$b_{\alpha,0}^{d,k}+b_{-\alpha,0}^{d,k}=2b_{0,0}^{d,k}.$
In the equation (\ref{n12}) replacing $\alpha$ with $-\alpha$
\begin{equation} \label{n13}
 (d+\alpha)b_{-\alpha,0}^{d,k}+(k-n)b_{-\alpha,0}^{d,k-n}=
  db_{0,0}^{d,k}+(k-n)b_{0,0}^{d,k-n}
\end{equation}
and adding the equations (\ref{n12}) and (\ref{n13})  we deduce 
$b_{\alpha,0}^{d,k}=b_{0,0}^{d,k},$ for all $\alpha \in G.$
Putting $\beta=-\alpha,$ $j=0$ and 
$\beta=j=0$ in (\ref{n11})  we have 
\begin{equation}\label{n23}
(d+\alpha)b_{\alpha,i}^{d,k}-4\alpha b_{0,i}^{d,k}-2ib_{0,i+n}^{d,k}=(d-\alpha) b_{0,0}^{d,k-i}-(k-n)b_{\alpha,i}^{d,k-n}  +(k-i-n)b_{0,0}^{d,k-i-n}    
\end{equation}
\begin{equation}\label{n24}
(d-\alpha)b_{\alpha,i}^{d,k}-2ib_{\alpha,i+n}^{d,k}=d b_{0,0}^{d,k-i}-(k-n)b_{\alpha,i}^{d,k-n}  +(k-i-n)b_{0,0}^{d,k-i-n}.    
\end{equation}
Subtracting (\ref{n23}) from (\ref{n24}), it gives this relation
\begin{equation}\label{n25}
4\alpha b_{0,i}^{d,k}-2\alpha b_{\alpha,i}^{d,k}-\alpha b_{0,0}^{d,k-i}=2ib_{\alpha,i+n}^{d,k}-2ib_{0,i+n}^{d,k}.   \end{equation}
Next taking $\beta=i=0$ in (\ref{n11}) and substituting $j$ with $i$ on it with $\alpha=0$ in 
(\ref{n23}), then combining them we can get 
\begin{equation}\label{n26}
2\alpha b_{\alpha,i}^{d,k}-\alpha b_{0,0}^{d,k-i}=2ib_{\alpha,i+n}^{d,k}-2ib_{0,i+n}^{d,k}.   \end{equation}
Subtracting (\ref{n25}) from (\ref{n26}) we deduce this identity
$b_{\alpha,i}^{d,k}=b_{0,i}^{d,k},$
 from this with together (\ref{n26} we have 
 \begin{equation}\label{n28}
  2b_{0,i}^{d,k}=b_{0,0}^{d,k-i},   
 \end{equation}
for all $k\in \mathbb{Z}.$ Finally, by (\ref{n23}) and (\ref{n28}) we have 
$b_{0,0}^{d,k-i}=0$ for $k\in \mathbb{Z}.$

\end{enumerate}
Thus, it proves that $\varphi_d$ has the form 
\begin{longtable}{lclcl}
$\varphi_d(L_{\alpha,i})$&$=$&$\sum_{k\in\mathbb{Z}}a^{d,k-i}L_{\alpha+d,k}$&$=$&$\sum_{k\in\mathbb{Z}}a^{d,k}L_{\alpha+d,k+i}$\\ 
$\varphi_d(H_{\alpha,i})$&$=$&$\sum_{k\in\mathbb{Z}}a^{d,k-i}H_{\alpha+d,k}$&$=$&$\sum_{k\in\mathbb{Z}}a^{d,k}H_{\alpha+d,k+i}.$
\end{longtable}

\end{proof}

\begin{theorem} There are no non-trivial transposed Poisson  structures defined on $HW_n(G).$
\end{theorem}
\begin{proof} 
 We aim to describe the multiplication $\cdot.$ By Lemma \ref{l1}, for every element  $X\in \ \{L_{\alpha, i}, \   H_{\alpha, i} \ | \ \alpha \in G, \ i \in \mathbb{Z}\},$ there is a related
$\frac{1}{2}$-derivation $\varphi_{l}$ of $HW_n(G)$ such that 
$\varphi_{X}(Y)=X\cdot Y.$ 
Thanks to Theorem \ref{main3}, we have a description of $\frac{1}{2}$-derivations  of $HW_n(G).$ 
Now we consider  the multiplication $L_{\alpha,i}\cdot H_{\beta,j} .$

\begin{longtable}{rcl}
$\varphi_{H_{\beta,j}}(L_{\alpha,i})\ =\ L_{\alpha,i} \cdot H_{\beta,j}$ & $=$&$H_{\beta,j}\cdot L_{\alpha,i}\ =\ \varphi_{L_{\alpha,i}}(H_{\beta,j}),$\\

$\sum\limits_{d\in G\setminus\{0\}}\sum\limits_{k\in\mathbb{Z}}a_{H_{\beta,j}}^{d,k-i}L_{\alpha+d,k}
$ & $=$ & $\sum\limits_{d\in G\setminus\{0\}}\sum\limits_{k\in\mathbb{Z}}a_{L_{\alpha,i}}^{d,k-j}H_{\beta+d,k}.$
\end{longtable}
\noindent
Thus, we get $a_{L_{\alpha,i}}^{d,k}=a_{H_{\beta,j}}^{d,k}=0$.
So, it gives  $L_{\alpha,i}\cdot L_{\beta,j}=H_{\alpha,i}\cdot H_{\beta,j}=0$ 
and $(HW_n(G),\cdot)$ is trivial. 

\end{proof}

\end{document}